\documentclass[11pt,reqno]{amsart} 

\usepackage{amssymb,latexsym}
\usepackage{cite} 
\usepackage{graphicx}
\usepackage[height=190mm,width=130mm]{geometry} 

\theoremstyle{plain}
\newtheorem{thm}{Theorem}
\newtheorem{lem}{Lemma}

\newtheorem{conj}{Conjecture}

\theoremstyle{definition}

\theoremstyle{remark}

\numberwithin{equation}{section} 

\begin{document}
\title[More on Comparison between $GA$ Index and $ABC$ Index]{More on Comparison between First Geometric-Arithmetic Index and Atom-Bond Connectivity Index} 

\author{Zahid Raza}
\address{National University of Computer and Emerging Sciences\\ Department of Mathematics\\ B-Block, Faisal Town\\ 54000 Lahore\\ Pakistan}
\email{zahid.raza@nu.edu.pk}

\author{Akhlaq Ahmad Bhatti}
\address{National University of Computer and Emerging Sciences\\ Department of Mathematics\\ B-Block, Faisal Town\\ 54000 Lahore\\ Pakistan}
\email{akhlaq.ahmad@nu.edu.pk}

\author{Akbar Ali}
\address{National University of Computer and Emerging Sciences\\ Department of Mathematics\\ B-Block, Faisal Town\\ 54000 Lahore\\ Pakistan}
\email{akbarali.maths@gmail.com}


\begin{abstract}
The first geometric-arithmetic (GA) index  and atom-bond connectivity (ABC) index are molecular structure descriptors which play a significant role in quantitative structure-property relationship (QSPR) and quantitative structure-activity relationship (QSAR) studies. Das and Trinajsti\'{c} [\textit{Chem. Phys. Lett.} \textbf{497} (2010) 149-151] showed that $GA$ index is greater than $ABC$ index for all those graphs (except $K_{1,4}$ and $T^{*}$, see Figure 1) in which the difference between maximum and minimum degree is less than or equal to 3. In this note, it is proved that $GA$ index is greater than $ABC$ index for line graphs of molecular graphs, for general graphs in which the difference between maximum and minimum degree is less than or equal to $(2\delta-1)^{2}$ (where $\delta$ is the minimum degree and $\delta\geq2$) and for some families of trees. Therefore, a partial solution to an open problem proposed by Das and Trinajsti\'{c} is given.

\end{abstract}




\maketitle

\section{Introduction}
Let $G=(V,E)$ denote a graph with vertex set $V(G)=\{v_{1}, v_{2}, \ldots, v_{n}\}$ and edge set $E(G)$ such that $|E(G)|=m$. Suppose that $d_{i}$ is the degree of a vertex $v_{i}\in V(G)$ \cite{1}. All the graphs considered in this study are simple, finite and undirected.

Topological indices are numerical parameters of a graph which are invariant under graph isomorphisms. They play a significant role in mathematical chemistry especially in the quantitative structure-property relationship (QSPR) and quantitative structure-activity relationship (QSAR) investigations \cite{2,3}. A whole class of topological indices is the ``geometric-arithmetic indices''whose general definition is as follows \cite{4,5}:
\[GA_{general} = GA_{general}(G) = \sum_{ij\in E(G)}\frac{\sqrt{Q_{i}Q_{j}}}{\frac{1}{2}(Q_{i}+Q_{j})} ,\]
where $Q_{i}$ is some quantity that can be associated with the vertex $v_{i}$ of the graph $G$ in a unique manner. The first geometric-arithmetic ($GA$) index was proposed by Vuki\v{c}evi\'{c} and Furtula \cite{4} by setting $Q_{i}$ as the degree $d_{i}$ of the vertex $v_{i}$ of the graph $G$:
\[GA(G)=\sum_{ij\in E(G)}\frac{\sqrt{d_{i}d_{j}}}{\frac{1}{2}(d_{i}+d_{j})} .\]
It has been demonstrated, on the example of octane isomers, that GA index is well-correlated with a variety of physico-chemical properties \cite{4}.
The details about mathematical properties of the $GA$ indices and their applications in QSPR and QSAR can be found in the survey \cite{6} reported by Das, Gutman and Furtula.

Estrada et al. \cite{7} proposed a topological index, known as the atom-bond connectivity (ABC) index of graph $G$,
which is abbreviated as ABC(G) and defined as
\[ABC(G)=\sum_{ij\in E(G)}\sqrt{\frac{d_{i}+d_{j}-2}{d_{i}d_{j}}} .\]
The $ABC$ index provides a good model for the stability of linear and branched
alkanes as well as the strain energy of cycloalkanes \cite{7,8}. Due to its physico-chemical applicability, the $ABC$ index has attracted significant attention from researchers in recent years and many mathematical properties of this index were reported. For instance, see the papers \cite{10,11,12,13,14,15,16,17}, more precisely the recent ones \cite{9,22,23,24} and references cited therein.

A graph having maximum vertex degree at most 4 is known as a \textit{molecular graph}. The \textit{line graph} $L(G)$ of a graph $G$ has the vertex set $V(L(G)) = E(G)$ where the two vertices of $L(G)$ are adjacent if and only if the corresponding edges of $G$ are adjacent; detailed properties of a line graph can be found in \cite{1}. Some possible chemical applications of line graphs of molecular graphs were discussed in \cite{19}. Das and Trinajsti\'{c} \cite{18} compared the $GA$ and $ABC$ indices for molecular graphs and for general graphs in which the difference between maximum and minimum degree is less than or equal to three. Recently, the current authors \cite{21} derived a relation between $GA$ index and $ABC$ index. In the present work, these two indices are compared for line graphs of molecular graphs, for general graphs in which the difference between maximum and minimum degree is less than or equal to $(2\delta-1)^{2}$ (where $\delta$ is the minimum degree and $\delta\geq2$) and for some families of trees.

\begin{figure}
    \centering
    \includegraphics[width=2.6in, height=1.0in]{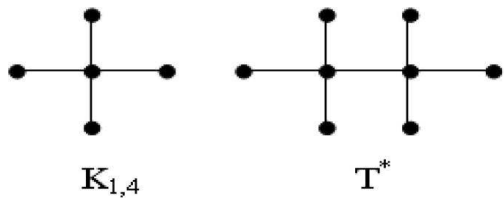}
    \caption{The molecular graphs $K_{1,4}$ and $T^{*}$}
    \label{fig:1}
\end{figure}

\section{Preliminaries}

The maximum and minimum vertex degree in a graph $G$ are denoted by $\Delta$ and $\delta$ respectively. Also, a vertex of the graph $G$ is said to be pendant if its neighborhood contains exactly one vertex. While, an edge of a graph is said to be pendant if one of its vertices is pendant.

The union $H\cup K$ of two graphs $H$ and $K$ is the graph with the vertex set $V(H)\cup V(K)$ and the edge set $E(H)\cup E(K)$. A tree in which exactly one of its vertices has degree greater than two is known as Starlike tree. Let $S(r_{1}, r_{2},\ldots, r_{k})$ denote the Starlike tree which has a vertex $v$ of degree $k>2$ such that the graph obtained from $S(r_{1}, r_{2},\ldots, r_{k})$ by removing the vertex $v$ is $P_{r_{1}}\cup P_{r_{2}}\cup \cdots \cup P_{r_{k}}$ where $P_{r_{i}}$ is the path graph on $r_{i}$ ($1\leq i\leq k$) vertices. We say that the Starlike tree $S(r_{1}, r_{2},\ldots, r_{k})$ has $k$ branches, the lengths of which are $r_{1}, r_{2},\ldots, r_{k}$ ($r_{1}\geq r_{2}\geq\cdots\geq r_{k}\geq1$), and has $\sum^{k}_{i=1}r_{i}+1$ vertices.

By a trivial graph, we mean a graph having one vertex. Denote by $K_{1,n}$ and $K_{n}$ the Star on $n+1$ vertices and complete graph on $n$ vertices respectively. A triangle of a graph $G$ is called odd if there is a vertex of $G$ adjacent to an odd number of its vertices.
\begin{lem}\label{L1} \cite{1}. A graph $G$ is a line graph if and only if $G$ does not have $K_{1,3}$ as an induced subgraph, and if two odd triangles have a common edge then the subgraph induced by their vertices is $K_{4}$.
\end{lem}
\section{Comparison between GA index and ABC index}

If the graph $G$ has $s\geq2$ components $G_{1},G_{2},\ldots,G_{s}$, then from the definition of the $GA$ and $ABC$ indices it follows that $ABC(G)=\sum_{i=1}^{s}ABC(G_{i})$ and $GA(G)=\sum_{i=1}^{s}GA(G_{i})$. Moreover, if the graph $G$ is trivial then $ABC(G)=GA(G)=0=ABC(P_{2})$. Hence it is enough to restrict our considerations to non-trivial and connected graphs only. Denoted by $T^{*}$ the tree on eight vertices, obtained by joining the central vertices of two copies $K_{1,3}$ by an edge (see Figure 1). To prove the first main theorem of this section, we need the following known result:

\begin{thm}\label{t1}
\cite{18}.
Let G be a non-trivial and connected graph with maximum degree $\Delta$ and minimum degree $\delta$. If $\Delta-\delta\leq3$ and $G\ncong K_{1,4}, T^{*}$, then $GA(G)>ABC(G)$.
\end{thm}

Let $m_{a,b}(G)$ be the number of edges of a graph $G$ connecting the vertices of degree $a$ and $b$. In the following theorem, we compare the $GA$ index and the $ABC$ index for line graph of a molecular graph:

\begin{thm}\label{t3}
Let $M$ be a molecular (connected) graph  with $n\geq3$ vertices and $G \cong L(M)$. Then $GA(G)>ABC(G)$.
\end{thm}
 \begin{proof}
If $n\leq4$ or $M\cong P_{n}$, then it can be easily seen that $G$ is a molecular graph which satisfies the hypothesis of Theorem \ref{t1} and hence the result follows. Let us assume that $n\geq5$ and $M\not\cong P_{n}$. Note that $1\leq d_{i}\leq6$ for all vertices $v_{i}$ of $G$. Hence the edges of $G$ are of possible degree pairs: $(6,6),(6,5),(6,4),(6,3),(6,2),(6,1),(5,5),(5,4),(5,3),(5,2),(5,1),(4,4),(4,3),\\(4,2),(4,1),(3,3),(3,2),(3,1),(2,2),(2,1)$. The values of $\theta_{ij}=\frac{2\sqrt{d_{i}d_{j}}}{d_{i}+d_{j}}$ and $\phi_{ij}=\sqrt{\frac{d_{i}+d_{j}-2}{d_{i}d_{j}}}$ for all above mentioned degree pairs are given in the Table 1 and Table 2 (Table 2 is taken from \cite{18}). From these tables one can note easily that
\begin{equation}\label{a}
\theta_{ij}-\phi_{ij}
\begin{cases}
       \geq \frac{2\sqrt{6}}{7}-\sqrt{\frac{5}{6}}\approx-0.2130 & \text{if } (d_{i},d_{j})=(4,1),(5,1),(6,1)\\
       \approx0.0495 & \text{if } (d_{i},d_{j})=(3,1)\\
       \approx0.1589 & \text{if } (d_{i},d_{j})=(6,2)\\
       \approx0.1964 & \text{if }(d_{i},d_{j})=(5,2)\\
       \approx0.2357 & \text{if } (d_{i},d_{j})=(2,1),(4,2)\\
       \geq \frac{2\sqrt{6}}{5}-\frac{1}{\sqrt{2}}\approx0.2727 & \text{otherwise }
\end{cases}
\end{equation}

It is claimed that
\begin{equation}\label{a1}
\displaystyle\sum^{6}_{b=2}m_{1,b}(G)\leq \lfloor \frac{|V(G)|}{2}\rfloor\leq \lfloor \frac{|E(G)|}{2}\rfloor.
\end{equation}
The right inequality obviously holds because $G$ contains atleast one cycle. To prove the left inequality it is enough to show that no two pendent edges of $G$ are adjacent. Contrarily, suppose that $e_{1}=uv$ and $e_{2}=uw$ are pendent edges of $G$. Since $n\geq5$, order of $G$ is at least 4. This implies that there exists a vertex $t$ (different from $v,w$) adjacent with $u$ in $G$. Then the graph obtained by removing all vertices except $u, v, w, t$ of $G$ is $K_{1,3}$ , a contradiction to the Lemma \ref{L1}. Now, we consider two cases:

\textit{Case 1.} If $m_{1,b}(G)=0$ for all $b\geq4$, then it follows from (\ref{a}) that $\theta_{ij}-\phi_{ij}>0$ for all edges $ij\in G$ and hence $GA(G)>ABC(G)$.

\textit{Case 2.} If $m_{1,b}(G)=0$ not for all $b\geq4$.\\
If $m_{2,5}(G)=m_{2,6}(G)=0$, then from (\ref{a}) and (\ref{a1}),  it follows that
\begin{equation}\label{b}
GA(G)-ABC(G)=\sum_{ij\in E(G)}(\theta_{ij}-\phi_{ij})>0
\end{equation}
\footnotesize

\begin{table}[h]
\begin{tabular}{|c|c|c|c|c|c|c|c|c|c|c|c|c|c|c|c|c|c|c|c|c|}\hline%
$(d_{i},d_{j})$&(6,6)&(6,5)&(6,4)&(6,3)&(6,2)&(6,1)&(5,5)&(5,4)&(5,3)&(5,2)&(5,1)\\\hline%
$\theta_{ij}$&1&$\frac{2\sqrt{30}}{11}$&$\frac{2\sqrt{6}}{5}$&$\frac{2\sqrt{2}}{3}$&$\frac{\sqrt{3}}{2}$&$\frac{2\sqrt{6}}{7}$
&1&$\frac{4\sqrt{5}}{9}$&$\frac{\sqrt{15}}{4}$&$\frac{2\sqrt{10}}{7}$&$\frac{\sqrt{5}}{3}$\\
$\phi_{ij}$&$\frac{1}{3}\sqrt{\frac{5}{2}}$&$\sqrt{\frac{3}{10}}$&$\frac{1}{\sqrt{3}}$&$\frac{1}{3}\sqrt{\frac{7}{2}}$&$\frac{1}{\sqrt{2}}$
&$\sqrt{\frac{5}{6}}$&$\frac{2\sqrt{2}}{5}$&$\frac{1}{2}\sqrt{\frac{7}{5}}$&$\sqrt{\frac{2}{5}}$&$\frac{1}{\sqrt{2}}$
&$\frac{2}{\sqrt{5}}$\\\hline
\end{tabular}
\caption{Values of $\theta_{ij}$ and $\phi_{ij}$ for all edges with degrees $(d_{i},d_{j})$ where $5\leq d_{i}\leq6$ and $d_{i}\geq d_{j}$}
\end{table}
\begin{table}[h]
\begin{center}
\begin{tabular}{|c|c|c|c|c|c|c|c|c|c|c|c|c|c|c|c|c|c|c|c|c|}\hline%
$(d_{i},d_{j})$&(4,4)&(4,3)&(4,2)&(4,1)&(3,3)&(3,2)&(3,1)&(2,2)&(2,1)\\\hline%
$\theta_{ij}$&1&$\frac{4\sqrt{3}}{7}$&$\frac{2\sqrt{2}}{3}$&0.8&1&$\frac{2\sqrt{6}}{5}$&$\frac{\sqrt{3}}{2}$
&1&$\frac{2\sqrt{2}}{3}$\\
$\phi_{ij}$&$\frac{1}{2}\sqrt{\frac{3}{2}}$&$\frac{1}{2}\sqrt{\frac{5}{3}}$&$\frac{1}{\sqrt{2}}$&$\frac{\sqrt{3}}{2}$&$\frac{2}{3}$
&$\frac{1}{\sqrt{2}}$&$\sqrt{\frac{2}{3}}$&$\frac{1}{\sqrt{2}}$&$\frac{1}{\sqrt{2}}$\\\hline
\end{tabular}
\end{center}
\caption{Values of $\theta_{ij}$ and $\phi_{ij}$ for all edges with degrees $(d_{i},d_{j})$ where $2\leq d_{i}\leq4$ and $d_{i}\geq d_{j}$}
\end{table}

\normalsize

If at least one of $m_{2,5}(G),m_{2,6}(G)$ is nonzero. Consider the edge $e=xy\in G$ where degree of $x$ and $y$ is two and $c$ $(c=5, 6)$ respectively. Let $l$ denote number of vertices of degree two which are adjacent with $y$. Then $1\leq l\leq2$ for otherwise $K_{1,3}$ would be an induced subgraph of $G$. Note that the vertex $y$ lies on either of the cliques $K_{c-1}, K_{c}$ of $G$ 
and hence the edges with possible degree pairs of these cliques in $G$ are $(6,6),(6,5),(6,4),(6,3),(5,5),(5,4),(5,3),(4,4),\\(4,3),(3,3)$. For all these degree pairs $\theta_{ij}-\phi_{ij}\geq0.2727$. Moreover, corresponding to every clique $K_{d}$ $(d=4, 5, 6)$ of $G$, there exist at most $2d$ edges with degree pairs $(2,c)$ in $G$, where vertex of degree $c$ (that is $y$) lies on $K_{d}$. Since the size of $K_{d}$ is $\frac{d(d-1)}{2}\geq 2d$ for $d\geq5$. Therefore, if $G$ does not have clique $K_{4}$, then by using (\ref{a}) and (\ref{a1}) one can easily see that the inequality (\ref{b}) holds. If $G$ has clique $K_{4}$. It can be easily seen that no edge with degree pairs $(2,6)$ can be incident with any vertex of the clique $K_{4}$. This implies, corresponding to every clique $K_{4}$ there exist at most 8 edges with degree pairs $(2,5)$, but the size of $K_{4}$ is 6. Hence
\[\frac{8(0.1964)+6(0.2727)}{14}\approx 0.2291>0.2130.\]
This completes the proof.
\end{proof}

Now, we prove that conclusion of Theorem 3.2 remains true if minimum degree is $k\geq2$ and the difference between maximum and minimum degree is less than or equal to $(2k-1)^{2}$. To proceed, we need the following lemma:

\begin{lem}\label{L3}
 If $f(x,y)=(x+y)^{2}x^{2}-(x+\frac{y}{2})^{2}(2x+y-2)$ , $k\leq x\leq k+(2k-1)^{2}$ and $0\leq y\leq (2k-1)^{2}$ where $k\geq2$
then $f(x,y)>0$.
\end{lem}

\begin{proof}
 \textbf{Step 1.} Firstly we take $x=k$, then
\[g(y)=f(k,y)=(k+y)^{2}k^{2}-(k+\frac{y}{2})^{2}(2k+y-2)\]
and $g'(y)>0$ implies that $\alpha<y<\beta$, where
\[\alpha=\frac{2}{3}(1-3k+2k^{2})-\frac{2}{3}\sqrt{1+4k^{2}-6k^{3}+4k^{4}}\]
and
\[\beta=\frac{2}{3}(1-3k+2k^{2})+\frac{2}{3}\sqrt{1+4k^{2}-6k^{3}+4k^{4}}.\]
It means that $g$ is increasing in the interval $(\alpha,\beta)$ and decreasing in the intervals $(-\infty,\alpha)$ and $(\beta,\infty)$.
Since $\alpha<0<\beta<(2k-1)^{2}$, so $g$ is increasing in $(0,\beta)$ and decreasing in $(\beta,(2k-1)^{2})$.
Moreover, $g(0)=k^{2}\{k^{2}-2(k-1)\}>0$ and $g((2k-1)^{2})=k^{4}+k^{2}-\frac{k}{2}+\frac{1}{4}>0$. It follows that $g(y)>0$ for all $y\in[0,(2k-1)^{2}]$.\\
\\\textbf{Step 2.} Now we take $y=y_{0}$, where $y_{0}$ is any fixed integer in the interval $[0,(2k-1)^{2}]$.
Let $h(x)=f(x,y_{0})$ then
\[h'(x)=(2x+y_{0})\left[(2x^{2}+2-3x)+\left(2x-\frac{3}{2}\right)y_{0}\right]>0\]
for all $x\geq k\geq 2$. Hence $h(x)=f(x,y_{0})$ is increasing in $[k,\infty)$. Combining both the results proved in Step 1 and Step 2, we have the lemma.
\end{proof}

\begin{thm}\label{t5}
Let $G$ be a connected graph with maximum degree $\Delta$ and minimum degree $\delta\geq2$. If $\Delta-\delta\leq(2\delta-1)^{2}$ then $GA(G)>ABC(G)$.
  \end{thm}

\begin{proof}
 Let us consider the quantity
\begin{equation}\label{1}
\Gamma=d_{i}^{2}d_{j}^{2}-\frac{1}{4}(d_{i}+d_{j})^{2}(d_{i}+d_{j}-2),
\end{equation}
 where $d_{i}$ and $d_{j}$ are the degrees of vertices $v_{i}$ and $v_{j}$ respectively in $G$. Since $\delta \leq d_{i},d_{j}\leq\Delta\leq \delta+(2\delta-1)^{2}$ this implies that $\mid d_{i}-d_{j}\mid\leq(2\delta-1)^{2}$. Without loss of generality we can suppose that $d_{i}\geq d_{j}$ then $d_{i}= d_{j}+\theta$ for some $\theta$ ; $0\leq\theta\leq (2\delta-1)^{2}$ and Eq.(\ref{1}) becomes
 \begin{equation}\label{2}
\Gamma=\left(d_{j}+\theta\right)^{2}d_{j}^{2}-\left(d_{j}+\frac{\theta}{2}\right)^{2}(2d_{j}+\theta-2),
\end{equation}
where $\delta\leq d_{j}\leq \delta+(2\delta-1)^{2}$ and $0\leq \theta\leq (2\delta-1)^{2}$. Now, from Lemma \ref{L3} and Eq.(\ref{2}), we have the desired result.
\end{proof}

If the condition $\Delta-\delta\leq(2\delta-1)^{2}$ is replaced by $\Delta-\delta\leq(2\delta-1)^{2}+1$ in Theorem \ref{t5}, then the conclusion may not be true. For instance, consider the complete bipartite graph $K_{r,s}$, if we take $r=\delta\geq2$ and $s=(2\delta-1)^{2}+\delta+1$ then

\begin{eqnarray}\nonumber
GA(K_{r,s})&=&\frac{2\left[\delta((2\delta-1)^{2}+\delta+1)\right]^{\frac{3}{2}}}{{(2\delta-1)^{2}+2\delta+1}}\\\nonumber
&<&\sqrt{\delta((2\delta-1)^{2}+2\delta-1)((2\delta-1)^{2}+\delta+1)}= ABC(K_{r,s}).
\end{eqnarray}
On the other hand, consider the graph $G$ obtained by joining any vertex of $K_{12}$ to a vertex of $K_{3}$ by
an edge. Then $\Delta=12,\delta=2$ which means that $\Delta-\delta=10>(2(2)-1)^{2}$, but $GA(G)>ABC(G)$. We have the following result:

\begin{thm}
If $G$ is a connected graph with minimum degree $\delta\geq2$ and $| d_{i}-d_{j}|\leq(2\delta-1)^{2}$ for all edges $ij\in E(G)$, then $GA(G)>ABC(G)$.
\end{thm}
\begin{proof}
The proof is similar to the proof of Theorem \ref{t5} and hence is omitted.
\end{proof}

A stronger version of the above result can be analogously proved:
 \begin{thm}
Let $G$ be a connected graph with minimum degree $\delta\geq2$ and $| d_{i}-d_{j}|\leq(2k-1)^{2}$ for all edges $ij\in E(G)$, where $k=\min\{d_{i},d_{j}\}$. Then $GA(G)>ABC(G)$.
\end{thm}

The current authors recently derived the following relation between $GA$ index and $ABC$ index:

\begin{thm}\cite{21}
Let $G$ be a connected graph and minimum degree $\delta\geq2$, then
\[\frac{\sqrt{2(n-2)}}{n-1}GA(G)\leq ABC(G)\leq\frac{n+1}{4\sqrt{n-1}}GA(G),\]
with left equality if and only if $G\cong K_{n}$ and right equality if and only if $G\cong C_{3}$.
\end{thm}

Let $\delta_{1}$ be the minimum non-pendant vertex degree in G. Now, we compare GA index and ABC index for trees.
\begin{thm}
 If $T$ is a tree with $n\geq3$ vertices such that $m_{1,b}=0$ for all $b\geq4$ and $\Delta-\delta_{1}\leq(2\delta_{1}-1)^{2}$, then $GA(T)>ABC(T)$.
 \end{thm}
\begin{proof}
Let us consider the difference
\begin{eqnarray}\nonumber
GA(T)-ABC(T)&=&\sum_{ij\in E(T)}(\theta_{ij}-\phi_{ij})\\\nonumber
&=&\sum_{ \substack{ ij\in E(T), \\
         d_{i}\neq1\neq d_{j}}}
         (\theta_{ij}-\phi_{ij})
+\sum_{\substack{ ij\in E(T), \\
         d_{i}=1 \text{ or } d_{j}=1}}
         (\theta_{ij}-\phi_{ij}).
\end{eqnarray}
As $m_{1,b}=0$ for all $b\geq4$, from (\ref{a}) it follows that
$$\sum_{\substack{ ij\in E(T), \\
         d_{i}=1 \text{ or } d_{j}=1}}
         (\theta_{ij}-\phi_{ij})>0$$
Now, we have to prove that
$$\sum_{ \substack{ ij\in E(T), \\
         d_{i}\neq1\neq d_{j}}}
         (\theta_{ij}-\phi_{ij})>0$$
To do so, let $d_{i},d_{j}\geq2$ then using the same technique, adopted in the proof of Theorem \ref{t5}, we have
$$d_{i}^{2}d_{j}^{2}-\frac{1}{4}(d_{i}+d_{j})^{2}(d_{i}+d_{j}-2)>0$$
which is equivalent to
$$\frac{2\sqrt{d_{i}d_{j}}}{d_{i}+d_{j}}>\sqrt{\frac{d_{i}+d_{j}-2}{d_{i}d_{j}}}$$
which implies that
$$\sum_{ \substack{ ij\in E(T), \\
         d_{i}\neq1\neq d_{j}}}\left(\frac{2\sqrt{d_{i}d_{j}}}{d_{i}+d_{j}}-\sqrt{\frac{d_{i}+d_{j}-2}{d_{i}d_{j}}}\right)>0.$$
This completes the proof.
\end{proof}
Now, for Starlike tree, we have the following result.
\begin{thm}\label{t9}
  Let $S=S(r_{1},r_{2},...,r_{k})$ be a Starlike tree.
\begin{enumerate}
\item If $r_{i}\geq4$ for all $i$, then $GA(S)>ABC(S)$.

\item If $r_{i}\geq2$ for all $i$ and $\frac{\sum\limits_{i=1}^kr_{i}}{k}\geq4$, then $GA(S)>ABC(S)$.

\item If $\frac{\sum\limits_{i=1}^kr_{i}}{k}\geq8$, then $GA(S)>ABC(S)$.
\end{enumerate}
\end{thm}

\begin{proof}
(1)
The edges of $S$ with possible degree pairs are: $(2,1),(2,2),(k,2)$. From Table 2 we have
\begin{equation}\label{3}
\theta_{ij}-\phi_{ij}\approx
\begin{cases}
0.2357 & \text{if } (d_{i},d_{j})=(2,1)\\
       0.2929 & \text{if } (d_{i},d_{j})=(2,2)
\end{cases}
\end{equation}
 Moreover, the function $f(k)=\theta_{2k}-\phi_{2k}=\left(\frac{2\sqrt{2k}}{k+2}-\frac{1}{\sqrt{2}}\ \right)$ is decreasing in $(2,\infty) $ and $f(k)\rightarrow -\frac{1}{\sqrt{2}}\approx-0.7071$ when $k\rightarrow\infty$. Hence we have $f(k)>-\frac{1}{\sqrt{2}}\approx-0.7071$ for all $k$. Since $r_{i}\geq4$ for all $i$, this implies that there are $k$ edges with degree pairs $(1,2)$, $k$ edges with degree pairs $(2,k)$ and at least $2k$ edges with degree pairs $(2,2)$ in $S$. This completes the proof of part $(1)$.

 Note that $\frac{2\sqrt{d_{i}d_{j}}}{d_{i}+d_{j}}-\sqrt{\frac{d_{i}+d_{j}-2}{d_{i}d_{j}}}> -1$ if $(d_{i},d_{j})=(1,k)$ for all $k$. Using the same technique, adopted in the proof of part $(1)$, one can easily prove parts $(2)$ and $(3)$.
\end{proof}

Let $W_{n}$ be the wheel graph of order $n$. Then
\[GA(W_{n})=(n-1)\left(1+\frac{2\sqrt{3(n-1)}}{n+2}\ \right)\]
and
\[ABC(W_{n})=(n-1)\left(\frac{2}{3}+\sqrt{\frac{n}{3(n-1)}}\ \right).\]
It can be easily verified that $GA(W_{n})>ABC(W_{n})$ for $4\leq n\leq194$ and $GA(W_{n})<ABC(W_{n})$ for $n\geq195$. Is there any graph $G$ with the property $GA(G)=ABC(G)$? All our attempts to find such a graph were unsuccessful. We end this section with following conjecture.
\begin{conj}
If $G$ is a non-trivial and connected graph, then $GA(G)\neq ABC(G)$.
\end{conj}

\section{Conclusion}
In \cite{18}, comparison between $GA$ index and $ABC$ index for general trees and general graphs was left as an open problem. Theorems \ref{t3} - \ref{t9} provide a partial solution of this open problem. The complete solution of the said problem remains a task for future.

\bibliography{MMN}

\begin{thebibliography}{10}
\expandafter\ifx\csname url\endcsname\relax
  \def\url#1{\texttt{#1}}\fi
\expandafter\ifx\csname urlprefix\endcsname\relax\def\urlprefix{URL }\fi

\bibitem{21}
\textsc{Ali, A., Bhatti, A., and Raza, Z.}: \emph{Further inequalities between
  vertex-degree-based topological indices}, arXiv:1401.7511 [math.CO].

\bibitem{23}
\textsc{Ashrafi, A.~R., Dehghan-Zadeh, T., Habibi, N., and John, P.~E.}:
  \emph{Maximum values of atom-bond connectivity index in the class of
  tricyclic graphs}, J. Appl. Math. Comput., \textbf{46} (2015),
  10.1007/s12190--015--0882--x.

\bibitem{13}
\textsc{Chen, J., Liu, J., and Guo, X.}: \emph{Some upper bounds for the
  atom-bond connectivity index of graphs}, Appl. Math. Lett., \textbf{25}
  (2012), No.~7, 1077--1081.

\bibitem{10}
\textsc{Das, K.}: \emph{Atom-bond connectivity index of graphs}, Discr. Appl.
  Math., \textbf{158}, No.~11, 1181--1188.

\bibitem{11}
\textsc{Das, K.~C., Gutman, I., and Furtula, B.}: \emph{On atom-bond
  connectivity index}, Chem. Phys. Lett., \textbf{511} (2011), No. 4-6,
  452--454.

\bibitem{6}
\textsc{Das, K.~C., Gutman, I., and Furtula, B.}: \emph{Survey on
  geometric-arithmetic indices of graphs}, MATCH Commun. Math. Comput. Chem.,
  \textbf{65} (2011), No.~3, 595--644.

\bibitem{18}
\textsc{Das, K.~C. and Trinajsti\'{c}, N.}: \emph{Comparison between first
  geometric-arithmetic index and atom-bond connectivity index}, Chem. Phys.
  Lett., \textbf{497} (2010), No. 1-3, 149--151.

\bibitem{2}
\textsc{Devillers, J. and Balaban, A.}: \emph{Topological Indices and Related
  Descriptors in QSAR and QSPR}, Gordon and Breach Science, Amsterdam, 1999.

\bibitem{22}
\textsc{Dimitrov, D.}: \emph{On structural properties of trees with minimal
  atom-bond connectivity index}, Discr. Appl. Math., \textbf{172} (2014),
  28--44.

\bibitem{8}
\textsc{Estrada, E.}: \emph{Atom-bond connectivity and the energetic of
  branched alkanes}, Chem. Phys. Lett., \textbf{463} (2008), No. 4-6, 422--425.

\bibitem{7}
\textsc{Estrada, E., Torres, L., Rodr\'{i}guez, L., and Gutman, I.}: \emph{An
  atom-bond connectivity index: modelling the enthalpy of formation of
  alkanes}, Indian J. Chem. A, \textbf{37} (1998), 849--855.

\bibitem{5}
\textsc{Fath-Tabar, G.~H., Furtula, B., and Gutman, I.}: \emph{A new
  geometric-arithmetic index}, J. Math. Chem., \textbf{47} (2010), No.~1,
  477--486.

\bibitem{14}
\textsc{Furtula, B., Graovac, A., and Vuki\v{c}evi\'{c}, D.}: \emph{Atom-bond
  connectivity index of trees}, Discr. Appl. Math., \textbf{157} (2009),
  No.~13, 2828--2835.

\bibitem{15}
\textsc{Gan, L., Hou, H., and Liu, B.}: \emph{Some results on atom-bond
  connectivity index of graphs}, MATCH Commun. Math. Comput. Chem., \textbf{66}
  (2011), No.~2, 669--680.

\bibitem{19}
\textsc{Gutman, I. and Estrada, E.}: \emph{Topological indices based on the
  line graph of the molecular graph}, J. Chem. Inf. Comput. Sci., \textbf{36}
  (1996), No. 1-3, 541--543.

\bibitem{3}
\textsc{Gutman, I. and Furtula, B.}: \emph{Novel Molecular Structure
  Descriptors—Theory and Applications}, vol. I-II, Univ. Kragujevac,
  Kragujevac, 2010.

\bibitem{1}
\textsc{Harary, F.}: \emph{Graph Theory}, Addison-Wesley, Reading, 1969.

\bibitem{9}
\textsc{Hosseini, S.~A., Ahmadi, M.~B., and Gutman, I.}: \emph{Kragujevac trees
  with minimal atom-bond connectivity index}, \textbf{71} (2014), No.~1, 5--20.

\bibitem{17}
\textsc{Lin, W., Gao, T., Chen, Q., and Lin, X.}: \emph{On the minimal abc
  index of connected graphs with given degree sequence}, MATCH Commun. Math.
  Comput. Chem., \textbf{69} (2013), No.~3, 571--578.

\bibitem{24}
\textsc{Palacios, J.~L.}: \emph{A resistive upper bound for the abc index},
  MATCH Commun. Math. Comput. Chem., \textbf{72} (2014), No.~3, 709--713.

\bibitem{4}
\textsc{Vuki\v{c}evi\'{c}, D. and Furtula, B.}: \emph{Topological index based
  on the ratios of geometrical and arithmetical means of end-vertex degrees of
  edges}, J. Math. Chem., \textbf{46} (2009), No.~4, 1369--1376.

\bibitem{16}
\textsc{Xing, R., Zhou, B., and Dong, F.}: \emph{On atom-bond connectivity
  index of connected graphs}, Discr. Appl. Math., \textbf{159} (2011), No.~15,
  1617--1630.

\bibitem{12}
\textsc{Zhou, B. and Xing, R.}: \emph{On atom-bond connectivity index}, Z.
  Naturforsch. A, \textbf{66} (2011), No. 1-2, 61--66.

\end{thebibliography}

\bibliographystyle{MMN}

\end{document}